\documentclass[11pt]{amsart}

\usepackage{mathrsfs,graphicx,latexsym,tikz,color,euscript}
\usepackage{amsfonts,amssymb,amsmath,amscd,amsthm}
\usepackage{epstopdf}
\usepackage{hyperref}
\usepackage{upgreek}
\usepackage{comment}
\usepackage{pdfsync}

\newtheorem{theorem}{Theorem}[section]
\newtheorem{lemma}[theorem]{Lemma}
\newtheorem{proposition}[theorem]{Proposition}
\newtheorem{cor}[theorem]{Corollary}

\theoremstyle{definition}
\newtheorem{definition}[theorem]{Definition}

\theoremstyle{remark}
\newtheorem{remark}[theorem]{Remark}

\numberwithin{equation}{section}

\newcommand{\rr}{\mathbb R}
\newcommand{\cc}{\mathbb C}

\newcommand{\pp}{\mathbb P}

\newcommand{\be}{\mathbb E}
\newcommand{\scc}{\mathscr C}

\newcommand{\al}{\alpha}

\newcommand{\gm}{\gamma}
\newcommand{\ep}{\varepsilon}

\newcommand{\tht}{\theta}
\newcommand{\om}{\omega}
\newcommand{\lmd}{\lambda}

\newcommand{\bn}{\mathbf n}
\newcommand{\bk}{\mathbf k}

\newcommand{\se}{\frac{2}{3}}

\begin{document}

\title[infinite spin]{The problem of infinite Spin for parabolic and collision solutions in the planar $n$-body problem}
 

\author{Zhe Wang}
\address{Chern Institute of Mathematics and LPMC, Nankai University, Tianjin, China}
\email{zhewang@mail.nankai.edu.cn}

\author{Guowei Yu}
\address{Chern Institute of Mathematics and LPMC, Nankai University, Tianjin, China}
\email{yugw@nankai.edu.cn}

\thanks{This work is supported by the National Key R\&D Program of China (2020YFA0713303), NSFC (No. 12171253), Nankai Zhide Foundation and the Fundamental Research Funds for the Central Universities.}

\begin{abstract} 
In the planar $n$-body problem, the problem of infinite spin occurs for both parabolic and collision solutions. Recently Moeckel and Montgomery \cite{MM25} showed that there is no infinite spin for total collision solutions, when the reduced and normalized configuration converges to an isolated central configuration. Following their approach, we show it can not happen for both complete and partially parabolic solutions, under similar conditions. Our approach also allows us to generalize Moeckel and Montgomery's result to partial collision solutions under similar conditions. 

\end{abstract}

\maketitle

\section{Introduction}
Consider the planar $n$-body problem with masses $m_i>0$, positions $q_i\in\mathbb{R}^2$, for each $i \in \bn$ $=\{1, 2, \dots, n\}$. The motion is governed by Newton's gravitational law
\begin{equation} \label{eq;n-body}
 m_i\ddot{q_i}=\nabla_iU(q), \; \forall i \in \bn,
\end{equation}
where $q=(q_i)_{i \in \bn} \in\mathbb{R}^{2n}$ and
$$  U(q)=\underset{i<j}{\sum}\frac{m_im_j}{r_{ij}}, \; r_{ij}=|q_i-q_j|. $$
Due to translation symmetry, we may always fix the center of mass at origin, i.e., 
$$ q \in \be :=\{ q \in \rr^{2n}: \sum_{i \in \bn} m_i q_i =0\}. $$

One of the most important aspects of the $n$-body problem is to understand the final motion of the masses as time goes to infinity. The fundamental work was done by Chazy \cite{Chazy22}, where he gave a complete classification of the final motion for the three body problem. For general $n$, it was studied by Pollard \cite{Pl67} and Saari \cite{Sr71}. In particular   Marchal and Saari obtained the following result in \cite[Corollary 1]{MS76}. 

\begin{theorem} \label{thm;March-Saari}
Given a solution $q(t)$ of \eqref{eq;n-body} with 
$$ R(t) = \max_{i \ne j} r_{ij}(t), \; r(t) = \min_{ i \ne j } r_{ij}(t)$$
then as $t \to \infty$, either $R(t)/t \to \infty$, or 
$$ q_i(t) = A_i t + O(t^{\se}), \; i \in \bn \text{ and } \limsup_{t \to \infty} r(t) >0 $$
where each $A_i \in \rr^2$ is a constant vector. 
\end{theorem}

The above result shows that when $R(t) = O(t)$, as $t \to \infty$, the entire system separates into subsystems, such that the distance between any two masses inside each subsystem is bounded by $O(t^{\se})$, while the distance between any two masses belong to different subsystem grows linearly as $t$. This shows that a critical case is when the distance between two masses grows in the order of $t^{\se}$. Such a critical case is called \emph{parabolic}, after parabolic solutions from the Kepler problem.

Motivated by this, we introduce the following definition.
\begin{definition}
\label{df;k-para} Given a $\bk \subset \bn$ with $|\bk| \ge 2$, we say a solution $q: (t_0, \infty) \to \be$ is $\bk$-parabolic, if there exist positive constants $C_i$, $i =1, 2, 3$, such that when $t \to \infty$, 
$$ 
\begin{cases}
C_1 t^{\se} \le r_{ij}(t) \le C_2 t^{\se}, & \; \forall i \ne j \in \bk; \\
r_{ij}(t) \ge C_3 t, & \; \forall i \in \bk, \forall j \in \bk':= \bn \setminus \bk.
\end{cases} 
$$
\end{definition}

\begin{remark}
For any subset $\bk$ of indices, $|\bk|$ denotes its cardinality. Notice that in the above definition, we do not assume $R(t) = O(t)$, as $t \to \infty$.
\end{remark}

 A $\bk$-parabolic solution is partially parabolic, if $|\bk| < n$, and complete parabolic, if $|\bk|=n$. Notice that there may be more than one subsystem, such that inside it the motion is parabolic. For results about existence of total and partial parabolic solutions with arbitrary initial configuration, we refer the readers to \cite{MV09} and \cite{PT24}.

If the motion of a subsystem is parabolic, the normalized configuration of masses in the subsystem must approach the set of central configuration. To explain this in details, we introduce the following notations. 

Given a subset $\bk \subset \bn$, we denote the center of mass of the $\bk$-subsystem as
$$ c_{\bk} = m^{-1}_{\bk} \sum_{i \in \bk} m_i q_i, \text{ where } m_{\bk} = \sum_{i \in \bk} m_i,$$
the relative configuration of the $\bk$-subsystem with respect to $c_{\bk}$ as 
$$ q^c_{\bk}= (q_i - c_{\bk})_{i \in \bk} \in \be_{\bk} :=\{ (q_i)_{i \in \bk} \in \rr^{2 |\bk|}: \sum_{i \in \bk} m_i q_i =0\}, $$
the moment of inertial of the $\bk$-subsystem with respect to $c_{\bk}$ as 
$$ I_{\bk}(q^c_{\bk})= \sum_{i \in \bk} m_i|q_i - c_{\bk}|^2,$$
and the normalized relative configuration with respect to $c_{\bk}$ as
$$ \hat{q}^c_{\bk}  = q^c_{\bk}/ \sqrt{I_{\bk}(q^c_{\bk})}  \in S_{\bk}, $$
where 
$$ S_{\bk} := \{ q_{\bk} \in \be_{\bk}: I_{\bk}(q_{\bk})=1 \}.$$
 
The energy of the $\bk$-subsystem is defined as 
$$ h_\mathbf{k}=K_\mathbf{k}(q^c_{\bk})-U_\mathbf{k}(q^c_{\bk}) := \underset{i\in\mathbf{k}}{\sum}\frac{m_i}{2}|\dot{q_i}- \dot{c}_{\bk}|^2 -\underset{i,j\in\mathbf{k},i<j}{\sum}\frac{m_im_j}{r_{ij}}.$$
The potential function between the $\bk$-subsystem and the rest of the masses is denoted by
\begin{equation*}
    U_{\mathbf{k},\mathbf{k^\prime}}=U-U_{\mathbf{k}}-U_{\mathbf{k^\prime}}=\underset{i\in\mathbf{k},j\in\mathbf{k^\prime}}{\sum}\frac{m_im_j}{r_{ij}}. 
\end{equation*}

\begin{definition}
\label{df;cc} $q_{\bk} = (q_i)_{i \in \bk} \in \be_{\bk}$ is a central configuration or CC of the $\bk$-subsystem, if there is a constant $\lmd >0$, such that 
$$ \nabla_{i} U_{\bk}(q_{\bk}) + \lmd m_i (q_i) = 0, \; \forall i \in \bk.$$
We denote the set of all CCs of the $\bk$-subsystem by $\mathscr{C}_{\bk}$ and the subset of normalized CCs as 
$$ \hat{\mathscr{C}}_{\bk} = \mathscr{C}_{\bk} \cap S_{\bk}. $$
\end{definition}
\begin{remark}
	$\hat{\mathscr{C}}_{\bk}$ is precisely the set of critical points of partial potential $U_{\bk}$ restricted on $S_{\bk}$. 
\end{remark}

\begin{proposition}
\label{prop;para-energy} If $q(t)$ is a $\bk$-parabolic solution, when $t \to \infty$,
\begin{enumerate}
\item[(a).]  $|h_{\bk}(t)| = O(t^{-5/3})$;
\item[(b).]  the normalized relative configuration $\hat{q}^c_{\bk}(t)$ must approach $\hat{\mathscr{C}}_{\bk}$. 
\end{enumerate}
\end{proposition}

\begin{remark}
Under the assumption $R(t) = O(t)$, as $t \to \infty$, this was already proven in \cite{Sr71} and \cite{MS76}. Here as we do not make such an assumption, a proof will be given in Section \ref{sec:para-energy}. 
\end{remark}

Let $\om(\hat{q}^c_{\bk}(t))$ denote the $\om$-limit set of $\hat{q}^c_{\bk}(t)$. Proposition \ref{prop;para-energy} implies 
$$ \om(\hat{q}^c_{\bk}(t)) \subset \hat{\scc}_{\bk}.$$
Then naturally we can ask the following question:

\emph{Does the $\om$-limit set contain a point, or equivalent does $\hat{q}^c_{\bk}(t)$ converge to a particular CC, as $t \to \infty$? } 

Due to the rotational symmetry, if $q_{\bk} \in \hat{\mathscr{C}}_{\bk}$, so is $R(\tht) q_{\bk}= (R(\tht) q_i)_{i \in \bk}$, for any rotation $R(\tht) \in SO(2)$. Consider the $(2|\bk|-4)$-dim quotient manifold $S_{\bk}/SO(2) \simeq \cc \mathbb{P}(|\bk|-2)$. We write the quotient map 
$$ S_{\bk} \to  \cc \mathbb{P}(|\bk|-2); \; q_{\bk} \mapsto [q_{\bk}]. $$
Let the function on $\cc \pp (|\bk|- 2)$ induced by $U_{\bk}$ be $[q_{\bk}] \mapsto U_{\bk}([q_{\bk}])$. We say $q_{\bk}$ is an isolated (resp. nondegenerate or degenerate) CC, if $[q_{\bk}]$ is an isolated (resp. nondegenerate or degenerate) critical point of $U([q_{\bk}])$.

Following the above notation, if we assume $[\hat{q}^c_{\bk}(t)]$ converges to a single isolated CC, as $t \to \infty$, the $\om$-limit set of $q^c_{\bk}(t)$ could be a subset of the circle generated by the corresponding CC. In particular it may contain the entire circle as $q^c_{\bk}(t)$ could go around infinitely many times. This is known as the \emph{infinite spin problem} (see \cite{Wt41}). 

To justify the assumption that $[\hat{q}^c_{\bk}(t)]$ converges to an isolated CC. We recall the following well-known conjecture for CC, which is listed by Smale \cite{Sm98} as problem 6 in his list of problems for the 21st century : 

\emph{For the planar $n$-body with arbitrary choice of mass, there are only finitely many normalized CCs up to the rotational symmetry.}

It is not hard to show if the above conjecture is true, every CC must be isolated. Although it is widely believed the conjecture is true, so far results are only available for small $n$. For $n=3$, it is a classical result due to Euler and Lagrange, for $n=4$, it was proven by Hampton and Moeckel \cite{HM06}, and for $n=5$, it was proven by Albouy and Kaloshin \cite{AK12} for generic choice of mass. For $n=6$, there are some partial progresses made by Chang and Chen \cite{CC24}. When all masses are equal and $n \le 7$, related results can be found in \cite{MZ19}. 

Here we show there is no infinite spin for a parabolic solution. 
\begin{theorem}
	\label{them;no-spin-para} If $q(t)$ is a $\bk$-parabolic solution with the reduced and normalized relative configuration $[q^c_{\bk}(t)] \in S_{\bk}/SO(2)$ converging to an isolated CC, then the normalized relative configuration $q^c_{\bk}(t)$ converges to a particular CC in $\mathscr{C}_{\bk}$, and in particular there is no infinite spin. 
\end{theorem}

Our proof is based on a recent result by Moeckel and Montgomery \cite{MM25}, where the authors proved there is no infinite spin for a total collision solution in the planar $n$-body problem under a similar assumption. There are some well-known similarities between collision and parabolic solutions, which we shall explain below. 

\begin{definition}
	\label{df;par-coll} We say a solution $q: (0, T) \to \be$, is $\bk$-collision, for a $\bk \subset \bn$ with $|\bk| \ge 2$, if $\lim_{t \to T} q(t) = q^* \in \be$ with 
	$$ \begin{cases}
		q^*_i = q^*_j, & \forall i \ne j \in \bk; \\
		q^*_i \ne q^*_j, & \forall i \in \bk, j \in \bk'. 
	\end{cases}
     $$
\end{definition}
\begin{remark}
	A $\bk$-collision solution is partially collision, if $|\bk| < n$, and total collision, if $|\bk|= n$. Notice that there may be more than one collision occurring at the same time. 
\end{remark}

Given a $\bk$-collision solution $q(t)$, when the masses in the $\bk$-subsystem approach the collision, it is well-known (see \cite{Chazy18}, \cite{Wt41} or \cite{FT04}), the normalized relative configuration $\hat{q}^c_{\bk}(t)$ of the $\bk$-subsystem must also approach the set of normalized CCs, as $t \to T$. If we assume $[\hat{q}^c_{\bk}(t)]$ converges to a single isolated CC, the problem of whether $\hat{q}^c_{\bk}(t)$ converges to a particular CC also occurs, and so is the problem of infinite spin. 

Since the work of Chazy \cite{Chazy18}, many people had tried to solve the problem of infinite spin for collision solutions. Several published papers had claimed to solve this problem, but all turned out to be incomplete. In \cite{MM25}, Moeckel and Montgomery explained in details why all these proofs are not complete. 

For a total collision solution in the planar $n$-body problem, first in \cite{Yu24}, Yu showed there is no infinite spin, when the center manifold is one dimensional and two dimensional in some cases. Finally use \L ojasiewicz inequality \cite{Lj82}, Moeckel and Montgomery \cite{MM25} proved there is no infinite spin stated as in the following theorem. 
\begin{theorem}{\cite[Theorem 1.2]{MM25}}
	\label{thm;total-coll:label} Suppose $q(t)$ is a total collision of the planar $n$-body problem with $q(t) \to 0$, as $t \to T$ and suppose the corresponding reduced and normalized configuration $[\hat{q}(t)] \in S_{\bn}/ SO(2)$ converges to an isolated CC. Then the normalized configuration $\hat{q}(t)$ converges to a particular CC in $\scc_{\bn}$, and there is no infinite spin. 
\end{theorem}

\begin{remark}
The problem of infinite spin for collision solutions is interesting purely from a mathematical point of view, as it only make sense when the bodies are point masses. Meanwhile for parabolic solutions it may have some practical meaning, as it still make sense when the bodies are not point masses. 
\end{remark}

The story for the problem of infinite spin of parabolic solutions is more or less the same, where it was claimed to be solved for complete parabolic solution in \cite{SH81} and \cite{Sr84}. However the proofs are incomplete for the same reason explained by Moeckel and Montgomery in \cite{MM25}. 

Our main contributions are the following two: 

First, we find a modification of the McGehee coordinate that is suitable for the study of parabolic solution, see \eqref{eq;McGehee-para}. As is well-known the McGehee coordinate was first introduced in \cite{Mg74} to study triple collision solutions in the collinear three body problem, but also works for collision solutions in general $n$-problem, see \cite{Mk89} and \cite{Eb90}.

Second, compare to \cite{MM25}, where only total collision solutions are considered,  here we not only consider complete parabolic solutions, but also partially parabolic solutions. In the later case, the problem becomes time dependent, due to the effect of the masses out of the subsystem. To overcome this, we prove the following theorem, which seems to be interesting by itself.   

Consider a smooth time-dependent system  
    \begin{equation}
    \label{eq;nonauto}
        \dot{x}=f(x)+g(x,t), \; x \in \rr^d,
    \end{equation}
    with $g(x,t)$ bounded for every $x\in\mathbb{R}^d$ and $g(x,t)=O(|x|^2)$ uniformly, for $|x|$ small enough. Define the spectral gap of $Df(0)$ as 
    \begin{align*}
    \beta:=\min\{|\text{re}(\lambda)|:\lambda \text{ is\ an eigenvalue\ of}\ Df(0)\ \text{with\ non-zero\ real\ part} \}. 
\end{align*}

\begin{theorem} \label{thm;cen-mfd-time-dep}
   Assume $\mathcal{N} \subset \rr^d$ is an invariant submanifold of  \eqref{eq;nonauto} with $g|_{\mathcal{N}}\equiv0$ and $0 \in \mathcal{N}$ being an equilibrium point. If a solution $x(t)$ of \eqref{eq;nonauto} satisfies
    \begin{equation}
    \label{eq;con-nonauto}
        x(t)\rightarrow0, \text{ as } t \to \infty, \text{ and } e^{\alpha t}|g\big(x(t),t\big)|<\infty, \; \forall t \text{ large enough},
    \end{equation}
    for every $\alpha\in (0,\beta)$,
    then there is a solution $y(t)$ in $\mathcal{N}$ and constants $C>0$, $\eta>0$, such that
    \begin{align*}
        |y(t)-x(t)|\le Ce^{-\eta t}, \; \forall t >0. 
    \end{align*} 
\end{theorem}
A proof will be given in Section \ref{sec:appendix} following Bressan's proof of the center manifold theorem in \cite{Br03}. Theorem \ref{thm;cen-mfd-time-dep} also allows us to generalize Moeckel and Montgomery's result to partial collision solutions, as in this case the corresponding problem is time dependent as well.   
\begin{theorem}
	\label{thm;par-coll} If $q(t)$ is a $\bk$-collision solution with the reduced and normalized relative configuration $[\hat{q}^c_{\bk}(t)]$ converging to an isolated CC, when $t \to T$, then the normalized relative configuration $q^c_{\bk}(t)$ converges to a particular CC in $\mathscr{C}_{\bk}$, and in particular there is no infinite spin. 
\end{theorem}

Although a proof of the above result has been given in \cite{GSZ24} recently,  our approach seems to be a more natural extension of Moeckel and Montgomery and we feel it is worth to include it here.



\section{Coordinate Transformation And Equations of Motion} \label{sec;co-tran}
Without loss of generality, let's assume $q(t)$ be a $\bk$-parabolic solution, with $\mathbf{k}=\{1,...,k\}$ for some $2\le k \le n$. Define a linear map $T:\mathbb{R}^{2k}\rightarrow\mathbb{R}^{2k}$, as 
$$ T(q_1,...,q_k)=(z_1,...,z_{k-1},c_\mathbf{k}):=(z,c), $$
where $z_i=q_i-c_\mathbf{k}\in\mathbb{R}^2$ and $z=(z_1,...,z_{k-1}) \in \rr^{2k}$. Since the mutual distances can be expressed by $z$, the potential function $U_\mathbf{k}(z)$ is still an analytic homogeneous function on $\mathbb{R}^{2k-2}\backslash\triangle_\mathbf{k}$, where 
 $$ \triangle_\mathbf{k}=\{z \in \rr^{2k}:r_{ij}=0\ \text{ for some } \{ i\neq j\} \subset \bk \}. $$
\\ \indent Then we have a time-dependent Lagrangian 
\begin{equation*}
    L(z,\dot{z},c,\dot{c},t)=\frac{1}{2}\dot{z}^TM\dot{z}+\frac{1}{2}m_0|\dot{c}|^2+U_\mathbf{k}+U_{\mathbf{k},\mathbf{k}^\prime}
\end{equation*}
where $m_0=\sum^k_{i=1}m_i$ and
\begin{align*}
    (T^{-1})^Tdiag(m_1,m_1,...,m_k,m_k)T^{-1}=
    \begin{pmatrix}
        M & 0 \\
        0 & m_0 
    \end{pmatrix}.
\end{align*}
Notice that $U_{\bk, \bk'}$ is the time-dependent term, as it depends on $q_{\bk'}(t) =(q_i(t))_{i =k+1}^n$. 

We introduce a Hermitian mass metric on $\mathbb{C}^{k-1}$
\begin{align*}
    \langle\langle v,w\rangle\rangle_{\mathbb{C}}=\bar{v}^TMw,\ \ \ \ \ \ \forall v,w\in\mathbb{C}^{k-1},
\end{align*}
where $\bar{v}$ is the complex conjugate of $v$. The real part $\langle\langle v,w\rangle\rangle= \text{re}\langle\langle v,w\rangle\rangle_{\mathbb{C}}$ is just the mass inner product.\\

Set $\mathbb{R}^{2k} \simeq \mathbb{C}^{k}$. As in \cite{MM25}, a local coordinate change $\psi:\mathbb{R}^+\times\mathbb{S}^1\times\mathbb{C}^{k-2}\times\mathbb{C}\rightarrow\mathbb{C}^{k}$ is 
$$ (z,c)=\psi(r,\theta,s,c)=(re^{i\theta}\frac{(s,1)}{||(s,1)||},c).$$
Introduce velocity variables $\rho=\dot{r}$, $\omega=\dot{s}$, and 
$$  \Omega(s,\omega)=\text{im}\langle\langle(s,1),(\omega,0)\rangle\rangle_{\mathbb{C}},\ \ G(s,\omega)=\text{re}\langle\langle(s,1),(\omega,0)\rangle\rangle_{\mathbb{C}} $$

We get the Lagrangian in the new coordinates $(r, \rho, \tht, \dot{\tht}, s, \om, c, \dot{c})$ as
\begin{align*}
    L=\frac{1}{2}\rho^2+\frac{1}{2}r^2\dot{\theta}^2+\frac{r^2}{2}||\omega||_{FS}^2+\frac{r^2}{2}\frac{\Omega(x,\omega)^2}{||(s,1)||^4}+\frac{r^2\dot{\theta}\Omega(s,\omega)}{||(s,1)||^2}+\frac{1}{r}V_\mathbf{k}(s)+\frac{1}{2}m_0|\dot{c}|^2+U_{\mathbf{k},\mathbf{k}^\prime},
\end{align*}
where $||\omega|| = ||(\omega,0)||$, $V_{\mathbf{k}}(s)=||(s,1)||U_{\mathbf{k}}(s,1)$ and 
\begin{align*}
    ||\omega||_{FS}^2&=\frac{||(s,1)||^2||(\omega,0)||^2-|\langle\langle(s,1),(\omega,0)\rangle\rangle_{\mathbb{C}}|^2}{(||(s,1)||^2)^2}\\
    &=\frac{||\omega||^2}{||(s,1)||^2}-\frac{G(s,\omega)^2+\Omega(s,\omega)^2}{(||(s,1)||^2)^2}.
\end{align*}

We will also use the notation $F(s,\omega)=||\omega||_{FS}^2$, which is the local representation of the square of the Fubini-Study metric on the complex projective space.\\
\indent Reverting to real coordinates, we can write the Fubini-Study norm as $F(s,\omega)=\omega^TA(s)\omega$ where $A(s)$ is a positive-definite $(2k-4)\times(2k-4)$ matrix. Observe that $\frac{\Omega(s,\omega)}{||(s,1)||^2}$ is real linear with respect to $\omega$, thus we can also write $\frac{\Omega(s,\omega)}{||(s,1)||^2}=B(s)\omega$ where $B(s)$ is a $1\times(2k-4)$ matrix. Then we have the corresponding Euler-Lagrange equations
\begin{equation}
\label{eq;EL-eq}
\begin{aligned}
    \dot{r}=&\rho\\
    \dot{\rho}=&rF(s,\omega)-\frac{1}{r^2}V_\mathbf{k}(s)+\frac{\mu^2}{r^3}+\frac{\partial U_{\mathbf{k},\mathbf{k}^\prime}}{\partial r}\\
    \dot{\theta}=&\frac{\mu}{r^2}-\frac{\Omega(s,\omega)}{||(s,1)||^2}\\
    \dot{\mu}=&\frac{\partial U_{\mathbf{k},\mathbf{k}^\prime}}{\partial \theta}\\
    \dot{s}=&\omega\\
    \dot{\omega}=&\frac{1}{2}A^{-1}(s)\nabla F(s,\omega)+\frac{1}{r^3}A^{-1}(s)\nabla V_{\mathbf{k}}(s)-\frac{2\rho\omega}{r}-A^{-1}(s)DA(s)(\omega)\omega\\&+\frac{1}{r^2}A^{-1}(s)\nabla U_{\mathbf{k},\mathbf{k}^\prime}-\frac{\dot{\mu}}{r^2}A^{-1}(s)B(s)\\
    \ddot{c}=&\frac{1}{m_0}\frac{\partial U_{\mathbf{k},\mathbf{k}^\prime}}{\partial c}
    \end{aligned}
\end{equation}
where $\nabla$ denotes the Euclidean gradient or partial gradient with respect to $s$ and 
$$ \mu=L_{\dot{\theta}}=r^2\dot{\theta}+\frac{r^2\Omega(s,\omega)}{||(s,1)||^2}. $$
Notice that $\mu$ is just the angular momentum, which means
\begin{equation}
\label{eq;angular-mom}
    \mu=\dot{z}^TMJz=\text{re}\langle\langle\dot{z},iz\rangle\rangle_{\mathbb{C}}.
\end{equation}

Equation \eqref{eq;EL-eq} is obtained by the following computation
\begin{align*}
    \frac{\partial L}{\partial\omega}=r^2A(s)\omega+\frac{r^2\Omega(s,\omega)}{||(s,1)||^4}B(s)+\frac{r^2\dot{\theta}}{||(s,1)||^2}B(s)=r^2A(s)\omega+\mu B(s),
\end{align*}
\begin{align*}
     \frac{d}{dt}\frac{\partial L}{\partial \omega} =r^2A(s)\dot{\omega}+2r\rho A(s)\omega+r^2DA(s)(\omega)\omega+\mu DB(s)\omega+\dot{\mu}\omega,
\end{align*}
\begin{align*}
    \frac{\partial L}{\partial s}&=\frac{r^2}{2}\nabla F(s,\omega)+\frac{r^2\Omega(s,\omega)}{||(s,1)||^2}DB(s)\omega+r^2\dot{\theta}DB(s)\omega+\frac{1}{r}\nabla V_{\mathbf{k}}(s)+\nabla U_{\mathbf{k},\mathbf{k}^\prime}\\
    &=\frac{r^2}{2}\nabla F(s,\omega)+\mu DB(s)\omega+\frac{1}{r}\nabla V_{\mathbf{k}}(s)+\nabla U_{\mathbf{k},\mathbf{k}^\prime}.
\end{align*}

Meanwhile in the new coordinates, the energy of the $\bk$-subsystem is 
\begin{equation}
\label{eq;eng-eq}
    h_{\mathbf{k}}=\frac{\rho^2}{2}+\frac{\mu^2}{2r^2}+\frac{r^2}{2}F(s,\omega)-\frac{1}{r}V_{\mathbf{k}}(s).
\end{equation}

The following asymptotic estimates will be needed in our proofs. 
\begin{lemma}
\label{lem;est-orbit}
    Assume $(s(t),1)$ converges to an $(s_0, 1) \in \rr^{2k-2} \setminus \Delta_{\bk}$, as $t\rightarrow \infty$, then we have the following asymptotic estimates
    \begin{equation}
    \label{eq;est-para}
         \left| \frac{\partial U_{\mathbf{k},\mathbf{k}^\prime}}{\partial r}(t) \right|=O(t^{-2}),\ \  \left| \frac{\partial U_{\mathbf{k},\mathbf{k}^\prime}}{\partial\theta}(t) \right|=O(t^{-\frac{4}{3}}),\ \ |\nabla U_{\mathbf{k},\mathbf{k}^\prime}(t)|=O(t^{-\frac{4}{3}}).
    \end{equation}  
\end{lemma}
\begin{proof}
View $z_i=q_i-c_{\mathbf{k}},i=1,...,k$ as functions of $(r,\theta,s)$, then
    \begin{align*}
        \frac{\partial U_{\mathbf{k},\mathbf{k}^\prime}}{\partial r}(t)=\sum_{\substack{i \in \bk, j \in \bk'}}\langle-\frac{m_im_j(q_i(t)-q_j(t))}{|q_i(t)-q_j(t)|^3},\frac{\partial z_i}{\partial r}\rangle.
    \end{align*}
By Definition \ref{df;k-para}, $|q_i(t)-q_j(t)|^{-1}=O(t^{-1})$. Meanwhile our assumption implies $|\frac{\partial z_i}{\partial r}|=O(1)$, thus $\frac{\partial U_{\mathbf{k},\mathbf{k}^\prime}}{\partial r}(t)=O(t^{-2})$. The other two estimates can be obtained similarly, after noticing that $|\frac{\partial z_i}{\partial \theta}(t)|=O(r)=O(t^{\frac{2}{3}})$ and $|\nabla z_i(t)|=O(r)=O(t^{\frac{2}{3}})$.
\end{proof}

A direct corollary of the above lemma is 
\begin{cor}
\label{cor;mu-P-Q-bound} $|\mu(t)|=O(1)$,  $|P(t)|=O(1)$ and $|Q(t)|=O(1)$, as $t \to \infty$, where
$$
\begin{aligned}
P(t)& =\mu^2(t)+r^3(t)\frac{\partial U_{\mathbf{k},\mathbf{k}^\prime}}{\partial r}(t); \\
Q(t) & =r^2(t)\Big(A^{-1}(t)\nabla U_{\mathbf{k},\mathbf{k}^\prime}(t)-\frac{\partial U_{\mathbf{k},\mathbf{k}^\prime}}{\partial\theta}(t)A^{-1}(t)B(t)\Big).
\end{aligned}
$$ 
\end{cor}

To summarize we have showed $(r,\rho,s,\omega)(t)$ satisfies the following non-autonomous equation
\begin{equation}
\label{eq;EL-eq2}
\begin{aligned}
    \dot{r}=&\rho\\
    \dot{\rho}=&rF(s,\omega)-\frac{1}{r^2}V_\mathbf{k}(s)+\frac{1}{r^3}P(t)\\
    \dot{s}=&\omega\\
    \dot{\omega}=&\frac{1}{2}\tilde{\nabla} F(s,\omega)+\frac{1}{r^3}\tilde{\nabla} V_{\mathbf{k}}(s)-\frac{2\rho\omega}{r}-A^{-1}(s)DA(s)(\omega)\omega+\frac{1}{r^4}Q(t)
    \end{aligned}
\end{equation}
where $\tilde{\nabla}=A^{-1}(s)\nabla$ denotes the gradient with respect to the Fubini-Study metric.

By the third equation in \eqref{eq;EL-eq}, there is a constant $C>0$,
\begin{equation}
\label{eq;tht-dot} |\dot{\theta}|\le \frac{|\mu|}{r^2}+C||\omega||_{FS}.
\end{equation}
Since $|\mu(t)|=O(1)$ and $r(t)=O(t^{\frac{2}{3}})$, $\theta(t)$ converges to a limit, when $t\rightarrow+\infty$, if the Fubini-Study arclength $L(s)$ given as below is finite,  
\begin{equation}
\label{eq;len-eq}
    L(s)=\int||\dot{s}(t)||_{FS}\ dt=\int||\omega(t)||_{FS}\ dt.
\end{equation}

To better describe the asymptotic behavior of the solution, we introduce the following modification of the McGehee coordinate with the rescaled variables and time parameter defined as
\begin{equation} \label{eq;McGehee-para}
u=r^{-\frac{1}{2}}, \; v=\sqrt{r}\rho, \; w=r^{\frac{3}{2}}\omega, \; d \tau = r^{-\frac{3}{2}} dt
\end{equation}
Using $'$ to represent derivatives with respect to $\tau$, then $\gm(\tau)=(u,v,s,w)(\tau)$ satisfies equation
\begin{equation}
\label{eq;res-eq}
    \begin{aligned}
        u^\prime=&-\frac{1}{2}uv\\
        v^\prime=&\frac{v^2}{2}+F(s,w)-V_{\mathbf{k}}(s)+u^2P(\tau)\\
        s^\prime=&w\\
        w^\prime=&-\frac{vw}{2}+\tilde{\nabla}V_{\mathbf{k}}(s)+\frac{1}{2}\tilde{\nabla}F(s,w)-A^{-1}(s)DA(s)(w)w+u^2Q(\tau)
    \end{aligned}
\end{equation}
where $P(\tau)=P(t(\tau))$, $Q(\tau)=Q(t(\tau))$.  

In the new variables, the energy of the $\bk$-subsystem is
\begin{equation}
\label{eq;eng-eq2}
    h_{\mathbf{k}}=u^2\Big(\frac{v^2}{2}+\frac{u^2\mu^2}{2}+\frac{1}{2}F(s,w)-V_{\mathbf{k}}(s)\Big).
\end{equation}

\begin{lemma}
\label{lem;hk-est-2} $u^{-2}(\tau)|h_{\mathbf{k}}(\tau)| = o(1)$, as $\tau \to \infty$. 
\end{lemma}

\begin{proof}
 This follows directly from property (a) in Proposition \ref{prop;para-energy}. 
 \end{proof}  

\begin{remark}
    Suppose that the orbit $q(t)$ exists on $(-2\delta,+\infty)$ for some $\delta>0$. Since we only care about the behavior of the orbit when time goes to infinity, we can write the regularization more meticulously. Choose a smooth partition of unity $\{\psi_1,\psi_2\}$ subordinating to $\{(-\delta,0),(-\delta/2,+\infty)\}$. Then we define $\tilde{r}(t)=\psi_1(t)(t+\delta)^{-2}+\psi_2(t)r(t)$ which is a smooth positive function on $(-\delta,+\infty)$ and there exists a $t_0>0$ such that $\tilde{r}(t)=r(t)$ when $t\ge t_0$. Now the new time variable $\tau$ defined by $\frac{d\tau}{dt}=\tilde{r}^{-\frac{3}{2}}(t)$ satisfies $\tau(t)\rightarrow-\infty$ when $t\rightarrow-\delta$ and $\tau(t)\rightarrow+\infty$ when $t\rightarrow +\infty$. Particularly, $P(\tau)$ and $Q(\tau)$ are uniformly bounded. In this case, we still denote $(u(\tau),v(\tau),s(\tau),w(\tau))$ as the solution of \eqref{eq;res-eq} and it is just the original orbit after some $\tau_0$.
\end{remark}

\begin{theorem}
    Assume $s(t)$ converges to an $s_0$, when $\tau \to \infty$ with $\frac{(s_0, 1)}{\| (s_0, 1)\|}$ being an isolated CC, then $\gamma(\tau)$ converges to an equilibrium $(0,v_0,s_0,0)$ of \eqref{eq;res-eq} with $v_0=\sqrt{2V_{\mathbf{k}}(s_0)}$ and $\tilde{\nabla}V_{\mathbf{k}}(s_0)=0$.  
\end{theorem}
\begin{proof}
    By our assumption, as $\tau \to \infty$, $u(\tau)$ converges to zero and $s(\tau)$ converges to a critical point $s_0$ of $V_{\mathbf{k}}|_{\mathbb{C}P^{k-2}}$, i.e., $\tilde{\nabla}V_{\mathbf{k}}(s_0)=0$. By Lemma \ref{lem;hk-est-2}, 
    \begin{align*}
        0=\underset{\tau\rightarrow+\infty}{\limsup}\ \Big(\frac{v^2}{2}+\frac{u^2\mu^2}{2}+\frac{1}{2}F(s,w)-V_{\mathbf{k}}(s)\Big)(\tau)\ge\underset{\tau\rightarrow+\infty}{\limsup}\ \frac{v^2}{2}-V_{\mathbf{k}}(s_0).
    \end{align*}
   From \eqref{eq;res-eq} and \eqref{eq;eng-eq2},  
   \begin{align*}
       v^\prime=2u^{-2}h_{\mathbf{k}}-\frac{v^2}{2}+V_{\mathbf{k}}(s)-u^2\mu^2+u^2P(\tau).
   \end{align*}
   We claim that for every sufficiently small $\ep>0$, there exists a $\tau_0>0$ such that for any $\tau > \tau_0$, 
    \begin{align*}
        \frac{v^2}{2}\ge V_{\mathbf{k}}(s_0)-2\ep\ \ and\ \ v>0.
    \end{align*}
    This implies $\underset{\tau\rightarrow+\infty}{\liminf}\ \frac{v^2}{2}\ge V_{\mathbf{k}}(s_0)$. We can choose $\tau_0>0$ such that
   \begin{align*}
       v^\prime\ge -\frac{v^2}{2}+V_{\mathbf{k}}(s_0)-\ep,\ \ \ \tau\ge\tau_0
   \end{align*}
   Now if $-\sqrt{2(V_{\mathbf{k}}(s_0)-2\ep)}<v<\sqrt{2(V_{\mathbf{k}}(s_0)-2\ep)}$, we have $v^\prime>\ep$. Since $u(\tau)$ converges to zero, $v$ cannot be less than or equal to zero all the time. Therefore, when $\tau$ is sufficiently large, $v$ will enter this range. Then $v(\tau)$ would increase beyond $\sqrt{2(V_{\mathbf{k}}(s_0)-2\ep)}$ and never decrease below it again. Combined with the first inequality in the proof, we derive $\underset{\tau\rightarrow+\infty}{\lim}\ v(\tau)=\sqrt{2V_{\mathbf{k}}(s_0)}$. Then from the energy equation, we obtain that $\underset{\tau\rightarrow+\infty}{\lim}\ w(\tau)=0$.
\end{proof}

We say $\gm(\tau)$ converges to an isolated (resp. nondegenerate or degenerate) equilibrium, if the corresponding $\frac{(s_0, 1)}{\| (s_0, 1)\|}$ is an isolated (resp. nondegenerate or degenerate) CC.

\section{No infinite spin for parabolic solutions} \label{sec:no_infinite_spin_for_parabolic_solutions}

A proof of Theorem \ref{them;no-spin-para} will be given in this section. Let's assume $q(t)$ is a $\bk$-parabolic solution and $\gm(\tau)=(u, v, s,w)(\tau)$ is the corresponding solution of \eqref{eq;res-eq}, which converges to an isolated equilibrium $p_0=(0, v_0, s_0, 0)$ with $v_0>0$, as $\tau \to \infty$. Then the linearization of \eqref{eq;res-eq} at $p_0$ is
\begin{align} \label{eq;lin}
    \begin{bmatrix}
\delta u'\\
\delta v^\prime\\
\delta s^\prime\\
\delta w^\prime
\end{bmatrix}=
\begin{bmatrix}
   -v_0/2 & 0 & 0 & 0\\
    0 & v_0& 0 &0\\
    0 & 0 & 0 & I\\
    0& 0 & D\tilde{\nabla} V_\mathbf{k}(s_0) & -\frac{1}{2} v_0 I
\end{bmatrix}
\begin{bmatrix}
    \delta u\\
    \delta v\\
    \delta s\\
    \delta w
\end{bmatrix}
\end{align}
Let $B$ be the lower right $(4k-8)\times(4k-8)$ block. If $\delta s$ satisfies $D\tilde{\nabla} V_\mathbf{k}(s_0)\delta s=c\delta s$ then $(\delta s,\delta w)=(\delta s,\lambda_{\pm}\delta s)$ is an eigenvector of $B$ with eigenvalue
\begin{align*}
    \lambda_{\pm}=\frac{-v_0\pm\sqrt{v_0^2+16c}}{4}.
\end{align*}
Since $v_0>0$, it follows that any nonreal eigenvalues are stable. Also $\text{re}(\lambda_{-})<0$ and we have $\lambda_{+}=0$ if and only if $c=0$. Let $\beta$ be the spectral gap of this linearization matrix.

Notice that $\mathcal M = \{u=0\}$ is an invariant submanifold of \eqref{eq;res-eq}. By Corollary \ref{cor;mu-P-Q-bound}, \eqref{eq;res-eq} also satisfies the conditions required for \eqref{eq;nonauto}. This allows us to prove the following result using Theorem \ref{thm;cen-mfd-time-dep}.

\begin{proposition}
\label{prop;app-sol} Let $\gm(\tau)$ be the solution of \eqref{eq;res-eq} as above, then there is a solution $\hat{\gm}(\tau)$ of \eqref{eq;res-eq}, which is contained in $\mathcal M$ and satisfies 
$$|\gm(\tau) -\hat{\gm}(\tau)| \le C e^{-\eta \tau}, \; \forall \tau >\tau_0,$$
for some positive constant $C$ and $\eta$. 
\end{proposition}

\begin{proof}
Recall that $u^\prime =-\frac{1}{2}uv$ and $v(\tau)\rightarrow v_0$ with $v_0>0$, $u(\tau)$ converges to $0$ exponentially. This implies $u(\tau)<e^{-\frac{1}{4}v_0\tau}$, when $\tau$ is sufficiently large. 

Since $-\frac{v_0}{2}$ is an eigenvalue, when $\beta\le \frac{1}{2}v_0$, we have $e^{\al \tau}u^2(\tau)\rightarrow 0$ as $\tau\rightarrow+\infty$ for every $\al\in(0,\beta)$. By Corollary \ref{cor;mu-P-Q-bound}, both $P(\tau)$ and $Q(\tau)$ are bounded, hence $\gamma(\tau)$ satisfies the conditions in Theorem \ref{thm;cen-mfd-time-dep}. As a result, we can find a solution $\hat{\gamma}(\tau)$ contained in the submanifold $\mathcal M$, such that the distance between $\gamma(\tau)$ and $\hat{\gamma}(\tau)$ converges to zero exponentially. Thus $\hat{\gamma}(\tau)\rightarrow p_0$ as $\tau\rightarrow +\infty$ as well.\\
\end{proof}
 
If the matrix $D\tilde{\nabla}V_\mathbf{k}(s_0)$ is nonsingular, the corresponding equilibrium $p_0$ is hyperbolic. Then any solution $\hat{\gm}(\tau)$ in $\mathcal M$ approaching $p_0$, as $\tau\rightarrow \infty$, are in the stable manifold and converge exponentially fast. From this, it follows that the integrand of the arclength integral \eqref{eq;len-eq} converges to $0$ and therefore $L(\hat{s})<\infty$. Then the asymptotic estimate of the distance between $\gamma(\tau)$ and $\hat{\gamma}(\tau)$ given by Proposition \ref{prop;app-sol} implies $L(s)<\infty$ as well. This proves Theorem \ref{them;no-spin-para}, when $p_0$ is nondegenerate.

Now assume $p_0$ is a degenerate equilibrium. Again by Proposition \ref{prop;app-sol}, we can find a solution $\hat{\gm}(\tau) = (\hat u, \hat v, \hat s, \hat w)(\tau)$ of \eqref{eq;res-eq} entirely contained in $\mathcal M$, which converges to $\gm(\tau)$ exponentially fast as $\tau \to \infty$. By the same argument as above, we have $L(s) < \infty$, once $L(\hat s) <\infty$ is proven. 

Since $\hat{\gm}(\tau) \in \mathcal M$, we only need to consider the restriction of \eqref{eq;res-eq} on $\mathcal M$, which is
\begin{equation}
\label{eq;res-eq-submnf}
    \begin{aligned}
        v^\prime=&\frac{v^2}{2}+F(s,w)-V_{\mathbf{k}}(s)\\
        s^\prime=&w\\
        w^\prime=&-\frac{vw}{2}+\tilde{\nabla}V_{\mathbf{k}}(s)+\frac{1}{2}\tilde{\nabla}F(s,w)-A^{-1}(s)DA(s)(w)w
    \end{aligned}
\end{equation}

As $p_0$ is degenerate, it has a local center manifold $\mathcal{W}_{\mathcal{U}}^c$ for some neighborhood $\mathcal{U}$, and by \eqref{eq;lin}, the center manifold (assume $\dim \mathcal{W}_{\mathcal{U}}^c=k$) is entirely contained in the invariant submanifold $ \mathcal M \cap \{\frac{v^2}{2}+\frac{1}{2}F(s,w)-V_{\mathbf{k}}(s)=0\}$. By choosing a suitable coordinate, ${W}_{\mathcal{U}}^c$ has the form of a graph
\begin{align*}
    {W}_{\mathcal{U}}^c=\{\big(0,-\sqrt{2V_{\bk}(s(x))-||w(x)||_{FS}^2},x,f(x),\phi(x),\psi(x)\big):x\in\mathcal{U}\subset\mathbb{R}^k\},
\end{align*}
where $\mathcal{U}$ is a neighborhood of the origin in $\mathbb{R}^k$ and $s(x)=(x,f(x))$, $w(x)=(\phi(x),\psi(x))$.\\
\indent Since \eqref{eq;res-eq-submnf} is time independent, by the center manifold theorem in \cite[page 330]{Br03}, we can find a solution $\tilde \gm(\tau)=(0, \tilde v, \tilde s, \tilde w)(\tau)$ of \eqref{eq;res-eq} contained in the center manifold, and a constant $\eta>0$, such that 
$$ e^{\eta \tau} |\hat{\gm}(\tau) -\tilde{\gm}(\tau)| \to 0, \; \text{ as } \tau \to \infty. $$
Therefore it is enough to show $L(\tilde s) < \infty$.\\
\indent We point out that \eqref{eq;res-eq-submnf} is exactly the same as the restriction of the vector field on the total collision manifold studied by Moeckel and Montgomery, see equation (3.1) in \cite{MM25}, by the same argument given in \cite[Section 4]{MM25}, we can pull-back the equations $\eqref{eq;res-eq-submnf}$ to $\mathcal{U}$ and get the pull-back differential equation on the center manifold, which can be approximated by a gradient. 
\begin{lemma}\label{lem;flow-cen-manf}
    The differential equation on the center manifold is $x^\prime=\phi(x)$ where $\phi(x)=k\tilde{\nabla}W(x)+\gamma(x)$ where $W(x)=V_{\bk}(x,\phi(x))$, $k=2/v(0)>0$ and $\gamma(x)=o(|\tilde{\nabla}W(x)|)$.
\end{lemma}
\begin{lemma}\label{lem;grd-ineq}
    In a sufficiently small neighborhood of the origin, the restricted potential $W(x)$ satisfies
    \begin{equation}\label{eq;grd-ineq}
        |\tilde{\nabla}W(x)|^2\ge |W(x)-W(0)|^\al
    \end{equation} 
    where $1<\al<2$.
\end{lemma}
The proofs of the above two lemmas are completely the same as Lemma 4.3 and 4.4 in \cite[Section 4]{MM25}, where the key is \L ojasiewicz inequality \cite{Lj82}.
\begin{lemma}\label{lem;grd-flow}
    Suppose $x(\tau)$ is a solution of a differential equation of the form $x^\prime=k\tilde{\nabla}W(x)+\gamma(x)$ where $k>0$ and $\gamma(x)=o(|\tilde{\nabla}W(x)|)$ and suppose that $W(x)$ satisfies an inequality of the form \eqref{eq;grd-ineq}. Suppose $x(\tau)$ is a solution with $x(\tau)\rightarrow0$ as $\tau\rightarrow +\infty$. Then the arclength of the curve $x(\tau)$ is finite. Here the gradient and arclength are taken with respect to some smooth Riemannian metric.
\end{lemma}
This lemma can be proven almost the same as Lemma 4.5 in \cite[Section 4]{MM25}, where the only difference is here $k>0$, while it is negative in \cite{MM25}, but the same proof goes through as well, since \eqref{eq;grd-ineq} still holds if we replace $W(x)$ by $-W(x)$.\\
\indent Lemma \ref{lem;grd-flow} then implies the following result.
 \begin{theorem}
\label{thm;flow-center-mnf}
    Suppose $p_0=(0, v_0, s_0, 0)$ is a degenerate restpoint with $v_0>0$. Let $\mathcal{U}$ be a sufficiently small neighborhood of $p_0$ and let $W_{\mathcal{U}}^c$ be the local center manifold. If $\tilde{\gamma}(\tau)$ is a solution in $W_{\mathcal{U}}^c$ which converges to $p_0$ as $\tau\rightarrow+\infty$, then $\tilde{\gamma}$ has finite arclength.
\end{theorem}
Since the finite arclength of $\tilde \gm$ implies $L(\tilde s)<\infty$, this finishes our proof of Theorem \ref{them;no-spin-para}.

\section{No infinite spin for collision solutions} \label{sec;coll-sol}
In this section, we apply our approach to partial collision solutions and give a proof of Theorem \ref{thm;par-coll} different from \cite{GSZ24}. Without loss of generality, let's assume $q: (0, T) \to \be$, be a $\bk$-collision solution, with $\bk=\{1, \dots, k\}$ and $2 \le k \le n$.

As the same set of coordinates $(r, \rho, \tht, \dot{\tht}, s, \om, c, \dot{c})$ given in Section \ref{sec;co-tran} still applies, we get the same time-dependent Euler-Lagrange equation given in \eqref{eq;EL-eq}. To obtain the corresponding asymptotic estimates for the time-dependent terms, we need the following lemmas.
\begin{lemma}
\label{lem;est-col-orbit}
    Assume $(s(t),1)$ converges to an $(s_0, 1) \in \rr^{2k-2} \setminus \Delta_{\bk}$, as $t\rightarrow T$, then we  have the following asymptotic estimates:
    \begin{align}
       \left| \frac{\partial U_{\mathbf{k},\mathbf{k}^\prime}}{\partial r}(t) \right|=O(1),\ \ \ \left|\frac{\partial U_{\mathbf{k},\mathbf{k}^\prime}}{\partial \theta}(t)\right|=O(r),\ \ \ |\nabla U_{\mathbf{k},\mathbf{k^\prime}}|=O(r).
    \end{align}
\end{lemma}
\begin{proof}
It follows from a similar argument as in Lemma \ref{lem;est-orbit}.\\
\end{proof}

\begin{lemma}
\label{lem;est-col-mu} Under the same condition as the above lemma, when $t \to T$, 
$$|\mu(t)|=O(r^\frac{5}{2}), \; |P(t)|=O(1), \; |Q(t)|=O(1), $$
where 
\begin{align*}
    P(t)& =\frac{\mu^2(t)}{r^3(t)}+\frac{\partial U_{\mathbf{k},\mathbf{k}^\prime}}{\partial r}(t),   \\
    Q(t)& =\frac{1}{r(t)}\Big(A^{-1}(t)\nabla U_{\mathbf{k},\mathbf{k}^\prime}(t)-\frac{\partial U_{\mathbf{k},\mathbf{k}^\prime}}{\partial\theta}(t)A^{-1}(t)B(t)\Big).
\end{align*} 
\end{lemma}
\begin{proof}
    Since $\mu$ is just the angular momentum of the $\bk$-subsystem as \eqref{eq;angular-mom}, there is a positive constant $C$, such that
    \begin{equation*}
        \mu^2\le C I_\mathbf{k}K_{\mathbf{k}}.
    \end{equation*} 
    Combining this with Sundman's estimates(see \cite{Sp70} or \cite[Proposition 6.25]{FT04}), 
    $$ I_{\mathbf{k}}(t)\sim |T-t|^\frac{4}{3}, \; K_{\mathbf{k}}(t)\sim |T-t|^{-\frac{2}{3}}, \; \text{ as } t \to T. $$
    We get 
    \begin{equation}
    \label{eq;mu-to-0} \mu(t) \to 0, \text{ as } t \to T. 
    \end{equation}

    Meanwhile by Lemma \ref{lem;est-col-orbit}, there is a constant $C>0$, such that for any $0<t_1<t_2<T$ 
    \begin{align*}
        |\mu(t_2)-\mu(t_1)|&\le\int^{t_2}_{t_1}|\dot{\mu}|\,dt =\int^{t_2}_{t_1}|\frac{\partial U_{\mathbf{k},\mathbf{k}^\prime}}{\partial \theta}(t)| \,dt\\
        &\le C\int^{t_2}_{t_1}(T-t)^\frac{2}{3}\,dt\\
        &=\frac{3}{5}C\Big((T-t_1)^\frac{5}{3}-(T-t_2)^\frac{5}{3}\Big).
    \end{align*} 
Together with \eqref{eq;mu-to-0}, they imply $|\mu(t)|=O(|t-T|^\frac{5}{3})=O(r^\frac{5}{2})$.

The rest of the lemma then follows from the estimate of $\mu(t)$ and Lemma \ref{lem;est-col-orbit}
\end{proof}

By Lemma \ref{lem;est-col-mu}  
$$ \int\frac{|\mu|}{r^2}\ dt<\infty. $$ 
Again by \eqref{eq;tht-dot}, $\theta(t)$ converges to a limit as $t\rightarrow T$, if 
\begin{align*}
    L(s)=\int ||\omega(t)||_{FS}\ dt<\infty.
\end{align*}

Using the McGehee coordinate, 
$$ v=\sqrt{r}\rho, \; w=r^{\frac{3}{2}}\omega, \; d\tau =r^{-\frac{3}{2}} dt. $$
We obtain the blown-up equation for $\gm(\tau) = (r, v, s, w)(\tau)$, 
\begin{equation} \label{eq;MC-coll}
    \begin{aligned}
        r^\prime=&rv\\
        v^\prime=&\frac{v^2}{2}+F(s,w)-V_{\mathbf{k}}(s)+r^2P(\tau)\\
        s^\prime=&w\\
        w^\prime=&-\frac{vw}{2}+\tilde{\nabla}V_{\mathbf{k}}(s)+\frac{1}{2}\tilde{\nabla}F(s,w)-A^{-1}(s)DA(s)(w)w+r^2Q(\tau)
    \end{aligned}
\end{equation}
Notice that the new time variable $\tau \to \infty$, when $t \to T$.

The assumption of Theorem \ref{thm;par-coll} implies $\gm(\tau)$ converges to an isolated equilibrium $p_0 =(0, v_0, s_0, 0)$ with $v_0 = -\sqrt{2V_{\bk}(s_0)} <0$ and $\tilde{\nabla}V_{\bk}(s_0)=0$.

Since $r^\prime=rv$ and $v\to v_0<0$, we have $r(\tau)<e^{\frac{1}{2}v_0\tau}$, for $\tau$ sufficiently large. As $v_0$ is an eigenvalue, when $\beta\le -v_0$, for any $\al\in(0,\beta)$, we have $e^{\al \tau}r^2(\tau)\rightarrow 0$ as $\tau\rightarrow+\infty$.  Like the parabolic case, $\mathcal M=\{ r =0\}$ is an invariant submanifold of \eqref{eq;MC-coll}, and by Lemma \ref{lem;est-col-mu}, \eqref{eq;MC-coll} satisfies the conditions required for \eqref{eq;nonauto}. Therefore by Theorem \ref{thm;cen-mfd-time-dep}, we can obtain a result similar to Proposition \ref{prop;app-sol}. Then Theorem \ref{thm;par-coll} can be proven by an argument similar to those given in Section \ref{sec:no_infinite_spin_for_parabolic_solutions}.


\section{Proof of Theorem \ref{thm;cen-mfd-time-dep}} \label{sec:appendix}
We shall introduce some notations first. Let $A=Df(0)$ and $h(x)=f(x)-Ax$. The space $\mathbb{R}^d$ can be decomposed as the sum of a stable, an unstable and a center subspace, respectively spanned by the eigenvectors of $A$ corresponding to eigenvalues with negative, positive and zero real part. We thus have 
\begin{align*}
    \mathbb{R}^d=V^{s}\oplus V^u\oplus V^c
\end{align*}
with continuous projections
\begin{align*}
    \pi_s:\mathbb{R}^d\rightarrow V^s,\ \ \ \ \pi_u:\mathbb{R}^d\rightarrow V^u,\ \ \ \ \pi_c:\mathbb{R}^d\rightarrow V^c,
\end{align*}
\begin{align*}
    x=\pi_s x+\pi_c x+\pi_u x.
\end{align*}    
Recall the spectral gap of $A$ is
\begin{align*}
    \beta =\min\{|\text{re}(\lambda)|:\lambda \text{ is\ an \ eigenvalue\ with\ non-zero\ real\ part} \}. 
\end{align*}
For every $\ep\in(0,\beta)$ there exists a constant $C_\ep$,  such that
\begin{equation}
\label{eq;est-decom}
    \begin{aligned}
    &||e^{At}\pi_s||\le C_\ep e^{-(\beta-\ep)t}\ \ \ \ t\ge 0,\\
    &||e^{At}\pi_u||\le C_\ep e^{(\beta-\ep)t}\ \ \ \ \ \ t\le 0,\\
    &||e^{At}\pi_c||\le C_\ep e^{\ep |t|}\ \ \ \ \ \ \ \ \ \ t\in\mathbb{R}.
\end{aligned}
\end{equation}

Following the proof of the center manifold theorem in \cite{Br03}, by using a cutoff function, we may assume $h(x)$ has a compact support and its $\mathcal{C}^{1}$ norm is arbitrarily small.\\
\indent Any solution $t\mapsto y(t)$ of \eqref{eq;nonauto} in $\mathcal{N}$ can be represented by the formula
\begin{align*}
    y(t)=e^A(t-t_0)y(t_0)+\int^{t}_{t_0}e^{A(t-\tau)}h\big(y(\tau)\big)\ d\tau.
\end{align*}
We can decompose this formula as a sum of its center, stable, unstable components. Notice that here we can choose different starting times in connection with different components:
\begin{equation}
\label{eq;var-decom1}
    \begin{aligned}
        y(t)=&\pi_s\Big(e^{A(t-t_0)}y(t_0)+\int^t_{t_0}e^{A(t-\tau)}h\big(y(\tau)\big)\ d\tau\Big)+\\
    &\pi_{cu}\Big(e^{A(t-t_{1})}y(t_{1})+\int^t_{t_{1}}e^{A(t-\tau)}h\big(y(\tau)\big)\ d\tau\Big).
    \end{aligned}
\end{equation}
Here and below $\pi_{cu}=\pi_c+\pi_u$ denotes the projection on the center-unstable space $V^c\oplus V^u$.\\
\indent Let $x:[0,+\infty)\mapsto\mathbb{R}^d$ be a solution of \eqref{eq;nonauto} which satisfies condition \eqref{eq;con-nonauto}. We extend $x(\cdot)$ to a bounded function $x^*(\cdot)$ defined on the whole real line by setting
$$ x^*(t) = \begin{cases}
x(t) & \text{ if}\ \ \ t \geq 0, \\
x(0) & \text{ if}\ \ \ t < 0.
\end{cases}
$$

Notice that $x^*$ provides a globally bounded solution to
\begin{align*}
    \frac{d}{dt} x^*(t)=Ax^*+h(x^*)+g(x^*,t)+\varphi(t)
\end{align*}
where
$$ 
\varphi(t)= \begin{cases}
0 &\ \ \ \ \text{if}\ \  t>0,\\
      -Ax(0)-h\big(x(0)\big)-g\big(x(0),t\big)&\ \ \ \ \text{if}\ \ t<0.
\end{cases}
$$
Although $x^*$ is not differentiable at the point 0, it still can be represented by the variation formula
\begin{equation}
\label{eq;var-decom2}
    \begin{aligned}
        x^*(t)&=e^{A(t-t_0)}\pi_sx^*(t_0)+\int^t_{t_0}e^{A(t-\tau)}\pi_sh\big(x^*(\tau)\big)\ d\tau\\
   &+\int^t_{t_0}e^{A(t-\tau)}\pi_{s}\Big[g\big(x^*(\tau),\tau\big)+\varphi(\tau)\Big]\ d\tau+e^{A(t-t_1)}\pi_{cu}x^*(t_1)\\
   &+\int^t_{t_1}e^{A(t-\tau)}\pi_{cu}h\big(x^*(\tau)\big)\ d\tau+\int^t_{t_1}e^{A(t-\tau)}\pi_{cu}\Big[g\big(x^*(\tau),\tau\big)+\varphi(\tau)\Big]\ d\tau.
    \end{aligned}
\end{equation}

Fix any number $\eta\in (0,\beta)$, we define a space of functions
$$ Z_\eta:=\{z\in C^0(\mathbb{R},\mathbb{R}^d):\ ||z||_\eta:=\underset{t}{\sup}\ e^{\eta t}|z(t)|<\infty\}. 
$$
Notice that $Z_\eta$ is a Banach space with the norm $||\cdot||_{\eta}$ given above. We claim that there exists a function $z\in Z_{\eta}$ such that $y=x^*+z$ is a global solution of \eqref{eq;nonauto} entirely contained in $\mathcal{N}$. From \eqref{eq;var-decom1} and \eqref{eq;var-decom2}, for any choice of $t_0$, $t_1$ such a function $z(\cdot)$ should provide a solution to the integral equation
\begin{align*}
    z(t)&=e^{A(t-t_0)}\pi_sz(t_0)+\int^t_{t_0}e^{A(t-\tau)}\pi_s\Big[h\big(x^*(\tau)+z(\tau)\big)-h\big(x^*(\tau)\big)\Big]\ d\tau\\
   &-\int^t_{t_0}e^{A(t-\tau)}\pi_{s}\Big[g\big(x^*(\tau),\tau\big)+\varphi(\tau)\Big]\ d\tau+e^{A(t-t_1)}\pi_{cu}z(t_1)\\
   &+\int^t_{t_1}e^{A(t-\tau)}\pi_{cu}\Big[h\big(x^*(\tau)+z(\tau)\big)-h\big(x^*(\tau)\big)\Big]\ d\tau\\
   &-\int^t_{t_1}e^{A(t-\tau)}\pi_{cu}\Big[g\big(x^*(\tau),\tau\big)+\varphi(\tau)\Big]\ d\tau.
\end{align*}
Letting $t_0\rightarrow-\infty$ and $t_1\rightarrow+\infty$, by \eqref{eq;est-decom} and the definition of $Z_{\eta}$, we obtain
\begin{equation}
    \begin{aligned}
        z(t)&=\int^t_{-\infty}e^{A(t-\tau)}\pi_s\Big[h\big(x^*(\tau)+z(\tau)\big)-h\big(x^*(\tau)\big)\Big]\ d\tau\\
   &-\int^t_{-\infty}e^{A(t-\tau)}\pi_{s}\Big[g\big(x^*(\tau),\tau\big)+\varphi(\tau)\Big]\ d\tau\\
   &-\int_t^{+\infty}e^{A(t-\tau)}\pi_{cu}\Big[h\big(x^*(\tau)+z(\tau)\big)-h\big(x^*(\tau)\big)\Big]\ d\tau\\
   &+\int_t^{+\infty}e^{A(t-\tau)}\pi_{cu}\Big[g\big(x^*(\tau),\tau\big)+\varphi(\tau)\Big]\ d\tau\\
   &:=\Lambda(z)(t).
    \end{aligned}
\end{equation}
From the definition of $x^*(t)$ and condition \eqref{eq;con-nonauto}, $g(x^*(t),t)\in Z_\eta$ and we denote $||g||_\eta:=||g(x^*(t),t)||_\eta$. Recalling that $\varphi(\tau)=0$ for $\tau>0$ and estimates \eqref{eq;est-decom}, then we have 
\begin{align*}
    e^{\eta t}|\Lambda(z)(t)|&\le \int^t_{-\infty}C_{\ep}e^{-(\beta-\ep)(t-\tau)}||h||_{C^1}e^{\eta(t-\tau)}||z||_{\eta}\ d\tau\\
   &+\int^t_{-\infty}C_{\ep}e^{-(\beta-\ep)(t-\tau)}e^{\eta(t-\tau)}||g||_\eta+e^{\eta t}\int^{\min\{0,t\}}_{-\infty}C_{\ep}e^{-(\beta-\ep)(t-\tau)}\sup_{t\in\mathbb{R}}|\varphi(t)|\ d\tau\\
   &+\int_t^{+\infty}C_{\ep}e^{\ep(\tau-t)}||h||_{C^1}e^{\eta(t-\tau)}||z||_{\eta}\ d\tau\\
   &+\int_t^{+\infty}C_{\ep}e^{\ep(\tau-t)}e^{\eta(t-\tau)}||g||_\eta+e^{\eta t}\int^0_{\min\{0,t\}}C_{\ep}e^{\ep(\tau-t)}\sup_{t\in\mathbb{R}}|\varphi(t)|\ d\tau\\
\end{align*}
If we let $\ep$ in \eqref{eq;est-decom} sufficiently small, there exist constants $C_1,C_2,C_3>0$ such that when $t>0$,
\begin{align*}
     e^{\eta t}|\Lambda(z)(t)|\le C_1||h||_{C^1}||z||_\eta+C_2||g||_\eta,
\end{align*}
and when $t<0$,
\begin{align*}
    e^{\eta t}|\Lambda(z)(t)|\le C_1||h||_{C^1}||z||_\eta+C_2||g||_\eta+C_3(e^{(\eta-\ep)t}-e^{\eta t}).
\end{align*}
Similarly choose $z_1, z_2 \in Z_{\eta}$, and then we have
\begin{align*}
    e^{\eta t}|\Lambda(z_1)(t)-\Lambda(z_2)(t)|\le C_1||h||_{C^1}||z_1 -z_2||_\eta.
\end{align*}
Thus the map $\Lambda:Z_{\eta}\rightarrow Z_{\eta}$ is a well-defined continuous map and is a strict contraction, provided that the norm $||h||_{\mathcal{C}^1}$ is suitably small. Therefore $\Lambda$ admits a unique fixed point $z\in Z_{\eta}$, which means $y=x^*+z$ represents a trajectory contained in $\mathcal{N}$. For all $t>0$ we now have
\begin{align*}
    |x(t)-y(t)|=|z(t)|\le e^{-\eta t}||z||_{\eta}.
\end{align*}


\section{Proof of Proposition \ref{prop;para-energy}} \label{sec:para-energy}
Let $q(t) = (q_i(t))_{i \in \bn}$ be a $\mathbf{k}$-parabolic solution. We set $z(t) = (z_i(t))_{i \in \bn}=(q_i(t)-c_{\mathbf{k}}(t))_{i \in \bn}$.
\begin{lemma}\label{lem;mut-dis}
    For each $k\in\mathbf{k}$, there is a $\gm_k(t)$ satisfying  $|\gamma_k(t)|=O(t^{-\frac{7}{3}})$ as $t\rightarrow+\infty$, and 
    \begin{equation}\label{eq;mutal-dist-eq}
        m_k\ddot{z}_k=\frac{\partial U_\mathbf{k}}{\partial z_k}+\gamma_k.
    \end{equation}
\end{lemma}
\begin{proof}
(a). When $k \in \bk$, by Newton's equation, 
\begin{equation}
 \label{eq;ddot-zk} \begin{aligned}
 m_k\ddot{z}_k&=\frac{\partial U_\mathbf{k}}{\partial z_k}+\sum_{j\notin\mathbf{k}}\frac{m_km_j(q_j-q_k)}{r_{kj}^3}-m_km_\mathbf{k}^{-1}\sum_{i\in\mathbf{k},j\notin\mathbf{k}}\frac{m_im_j(q_j-q_i)}{r_{ij}^3}\\
    &=\frac{\partial U_\mathbf{k}}{\partial z_k}+\sum_{j\notin\mathbf{k}}\frac{m_jm_k(q_j-c_{\mathbf{k}})}{r_{jk}^3}-\sum_{j\notin\mathbf{k}}\frac{m_jm_k(q_k-c_{\mathbf{k}})}{r_{jk}^3}-m_km_\mathbf{k}^{-1}\sum_{i\in\mathbf{k},j\notin\mathbf{k}}\frac{m_im_j(q_j-q_i)}{r_{ij}^3}.
 \end{aligned}
 \end{equation}  
For each $r_{ij}$ with $i\in\mathbf{k}, j\notin\mathbf{k}$, we have
$$ r_{ij}^2 = |z_i -z_j|^2 = |z_j|^2 \left( 1+ \frac{|z_i|^2 - 2 \langle z_i, z_j\rangle}{|z_j|^2} \right) := |z_j|^2( 1 + f_{ij}). $$
According to Definition \ref{df;k-para}, $|z_i(t)|=O(t^{\frac{2}{3}})$ and $|z_j(t)|\ge Ct$, as $t\rightarrow \infty$. Then
$$|f_{ij}(t)|=O(t^{-\frac{1}{3}}), \text{ when } t \to \infty.$$
This then implies 
\begin{equation*}
    \frac{1}{r_{ij}^3}=\frac{1}{z_{j}^3}(1-\frac{3}{2}f_{ij}+O(f_{ij}^2))=\frac{1}{z_j^3}(1+O(t^{-\frac{1}{3}})).
\end{equation*}
Plug these into \eqref{eq;ddot-zk}, we get  
\begin{align*}
    m_k\ddot{z}_k&=\frac{\partial U_\mathbf{k}}{\partial z_k}+\sum_{j\notin\mathbf{k}}\frac{m_jm_kz_j}{z_j^3}-\sum_{j\notin\mathbf{k}}\frac{m_jm_kz_k}{z_j^3}-m_k\sum_{j\notin\mathbf{k}}\frac{m_jz_j}{z_j^3}+O(t^{-\frac{7}{3}})\\
    &=\frac{\partial U_\mathbf{k}}{\partial z_k}+O(t^{-\frac{7}{3}}) 
\end{align*}
\end{proof}
Recall that $h_\mathbf{k}=\underset{i\in\mathbf{k}}{\sum}\frac{m_i}{2}|\dot{z}_i|^2 -U_{\bk}$. By differentiating both sides, we get   
$$ \dot{h}_{\bk} =\sum_{i \in \bk} \langle m_i\ddot{z}_i,\dot{z}_i\rangle-\sum_{i \in \bk}\langle \frac{\partial U_{\bk}}{\partial z_i},\dot{z}_i\rangle =\sum_{i \in \bk} \langle \gamma_i,\dot{z}_i\rangle.$$
From the Definition \ref{df;k-para} and Lemma \ref{lem;mut-dis}, we have $|\dot{h}_\bk|=O(t^{-\frac{8}{3}})$. Meanwhile Definition \ref{df;k-para} implies $h_\bk$ must converge to zero, thus $|h_\bk|=O(t^{-\frac{5}{3}})$, when $t \to \infty$.\\
\indent (b). With property (a), the desired result now follows directly from \cite[Corollary 4]{MS76}.

\hfill\newline
\noindent{\bf Acknowledgement.}  We thank Professors Alain Chenciner, Rick Moeckel, Richard Montgomery and Piotr Zgliczy\'nski for their interests and comments in our work. 

\bibliographystyle{abbrv}
\bibliography{Ref-Spin}

\end{document}